\documentclass{amsart}
\usepackage{amssymb}
\usepackage{graphicx}

\theoremstyle{plain}
\newtheorem{theorem}{Theorem}[section]

\newtheorem{conjecture}[theorem]{Conjecture}
\numberwithin{equation}{section}

\begin{document}

\title{An alternating sum of the floor function of square roots}

\author{Marc Chamberland}
\address{Department of Mathematics, Grinnell College,
Grinnell, IA 50112, USA}
\email{chamberl@grinnell.edu}
\author{Karl Dilcher}
\address{Department of Mathematics and Statistics,
         Dalhousie University,
         Halifax, Nova Scotia, B3H 4R2, Canada}
\email{dilcher@mathstat.dal.ca}
\thanks{Second author's research supported in part by the Natural Sciences and 
Engineering Research Council of Canada, Grant \# 145628481}

\date{}

\setcounter{equation}{0}

\begin{abstract}
We show that the alternating sum of the floor function of $\sqrt{jn}$, with $j$
ranging from 1 to $n$, has an easy evaluation for all odd integers $n\geq 1$.
This is in contrast to known non-alternating sums of the same type which hold
only for a class of primes. The proof is elementary and was suggested by an
AI model. To put this result in perspective, we also prove an asymptotic 
expression for the analogous sum without the floor function.
\end{abstract}

\maketitle

\section{Introduction}\label{sec:1}

The famous two-volume ``Problems and Theorems in Analysis" by P\'olya and 
Szeg{\H o} contains the following curious identity:
if $p\equiv 1\pmod{4}$ is a prime, then
\begin{equation}\label{1.1}
\sum_{j=1}^{\frac{p-1}{4}}\left\lfloor{\sqrt{jp}}\right\rfloor
= \frac{p^2-1}{12},
\end{equation}
where the {\it floor function} of $x\in{\mathbb R}$ is the integer 
$\lfloor{x}\rfloor$ defined by $\lfloor{x}\rfloor\leq x < \lfloor{x}\rfloor+1$.
Further remarks on this identity can be found in Section~\ref{sec:5} below.

It is a natural question to ask whether there are any similar results for 
{\it alternating} sums of the type \eqref{1.1}. While no such identity seems 
to exist for the summation limit $(p-1)/4$, numerical experiments led us to 
the following identity.

\begin{theorem}\label{thm:7.1}
For any odd integer $n\geq 1$ we have
\begin{equation}\label{7.1}
\sum_{j=1}^n(-1)^{j+1}\left\lfloor{\sqrt{jn}}\right\rfloor=\frac{n+1}{2}.
\end{equation}
\end{theorem}

This result is surprising for at least two reasons. 
First, it holds for {\it all} odd $n\geq 1$, and not just
for primes of a certain form, as is the case in \eqref{1.1}. 
Second, using standard methods from analytic number theory, we can prove
the following estimate, valid for any odd integer $n\geq 1$:
\begin{equation}\label{1.3}
\sum_{j=1}^n(-1)^{j+1}\sqrt{jn}=\frac{n}{2}+C\sqrt{n}+\frac{1}{8}
+ O\left(\frac{1}{n}\right)\qquad (n\rightarrow\infty),
\end{equation}
where $C\simeq 0.3801$ is a constant that can be explicitly expressed as an
infinite series. 

Comparing \eqref{1.3} with \eqref{7.1},
we see that by putting $\sqrt{jn}$ between the floor brackets, the error term
$C\sqrt{n}+1/8+O(1/n)$ turns into the surprisingly neat constant $1/2$.

It is the main purpose of this paper to prove the identity \eqref{7.1}; this is
done in Section~\ref{sec:2}. In Section~\ref{sec:4} we prove the estimate 
\eqref{1.3}, and we conclude this paper with a few remarks in 
Section~\ref{sec:5}.

\section{Proof of Theorem~\ref{thm:7.1}}\label{sec:2}

The following proof was suggested by the AI model {\it Gemini} \cite{Ge}. It
was then simplified and rewritten by us.

\begin{proof}[Proof of Theorem~\ref{thm:7.1}]
We rewrite the positive integer $\lfloor\sqrt{jn}\rfloor$ as a sum and then
change the order of summation, obtaining
\begin{align}
\sum_{j=1}^n(-1)^{j+1}\left\lfloor{\sqrt{jn}}\right\rfloor
&=\sum_{j=1}^n(-1)^{j+1}\sum_{k=1}^{\lfloor{\sqrt{jn}}\rfloor}1\label{2.1a}\\
&=\sum_{k=1}^n\sum_{j=\lceil k^2/n\rceil}^n(-1)^{j+1},\nonumber
\end{align}
where we have used the fact that $k\leq\sqrt{jn}$ is equivalent to 
$j\geq k^2/n$. Next, when $\lceil k^2/n\rceil$ is odd, then the first
summand in the inner sum on the right of \eqref{2.1a} is 1, followed by an equal
number of summands $-1$ and 1. When $\lceil k^2/n\rceil$ is even, the summands
all cancel each other, and therefore the inner sum is 1 when 
$\lceil k^2/n\rceil$ is odd and is 0 otherwise. Hence we have with 
\eqref{2.1a},
\begin{equation}\label{2.2a}
\sum_{j=1}^n(-1)^{j+1}\left\lfloor{\sqrt{jn}}\right\rfloor
= \#\{1\leq k\leq n\mid \lceil k^2/n\rceil\;\hbox{is odd}\}.
\end{equation}
where as usual $\#A$ denotes the cardinality of a finite set $A$.
When $k=n$, then $\lceil k^2/n\rceil=n$ is odd. Next, we note that
\[
\left\lceil\frac{(n-k)^2}{n}\right\rceil
=\left\lceil n-2k+\frac{k^2}{n}\right\rceil
=n-2k+\left\lceil\frac{k^2}{n}\right\rceil.
\]
Since $n$ is odd, $\lceil(n-k)^2/n\rceil$ and $\lceil k^2/n\rceil$
have opposite parities, and for $1\leq k\leq (n-1)/2$, they are exactly all
the remaining elements of the set in \eqref{2.2a}. So finally, by \eqref{2.2a}
we have
\[
\sum_{j=1}^n(-1)^{j+1}\left\lfloor{\sqrt{jn}}\right\rfloor
= 1 +\frac{n-1}{2} = \frac{n+1}{2},
\]
which was to be shown.
\end{proof}

To conclude this section, we briefly explain an alternative approach to a proof
of Theorem~\ref{7.1} which may be of interest in its own right.

The summand for $j=n$ in \eqref{7.1} is clearly 1, and we consider the
differences
\begin{equation}\label{2.3}
d_n(\ell):=\left\lfloor{\sqrt{2\ell n}}\right\rfloor
-\left\lfloor{\sqrt{(2\ell-1)n}}\right\rfloor,\qquad1\leq\ell\leq\frac{n-1}{2}.
\end{equation}
Then the statement of Theorem~\ref{thm:7.1} is equivalent to 
\begin{equation}\label{2.4}
\sum_{\ell=1}^{\frac{n-1}{2}}d_n(\ell)=\frac{n-1}{2}\qquad(n\;\hbox{odd}).
\end{equation}
The approach in question is now best explained by the following example. 
Let $n=33$. Table~1 shows the distribution of the values of $d_n(\ell)$.

\medskip
\begin{center}
\begin{tabular}{|c|cccccccccccccccc|}
\hline
$\ell$ & 1 & 2 & 3 & 4 & 5 & 6 & 7 & 8 & 9 & 10 & 11 & 12 & 13 & 14 & 15 & 16\\
\hline
$d_{33}(\ell)$ & 3 & 2 & 2 & 1 & 1 & 0 & 1 & 0 & 1 & 0 & 0 & 1 & 1 & 1 & 1 & 1  \\
\hline
\end{tabular}

\medskip
{\bf Table~1}: $d_{33}(\ell)$ for $1\leq\ell\leq 16$.
\end{center}
\bigskip

Here the value $d_{33}(1)=3$ corresponds to the two consecutive zeros for
$\ell=10$ and 11, while $d_{33}(2)=2$ corresponds to $d_{33}(8)=0$ and 
$d_{33}(3)=2$ to $d_{33}(6)=0$. This behavior is typical and is made more
explicit in the following statement.

\begin{conjecture}\label{conj:2.1}
Let $n\geq 1$ be an odd integer, and suppose that $\delta:=d_n(\lambda)\geq 2$
for some $\lambda\geq 1$. Then each such $\lambda$ corresponds to $\delta-1$
consecutive integers $\ell_{\lambda},\ldots,\ell_{\lambda}+\delta-2$ such that
$d_n(\ell_{\lambda}+j)=0$ for $j=0,\ldots,\delta-2$. Furthermore, the sets of
these zero values are disjoint.
\end{conjecture}

This conjecture implies \eqref{2.4} and thus Theorem~\ref{thm:7.1} since
each of the sequences of zeros is filled in by the ``surplus" given by the
corresponding value of $\delta$.

Our as yet incomplete proof of Conjecture~\ref{conj:2.1} relies on a sequence
of rather complicated technical lemmas. In view of the above complete proof
of Theorem~\ref{thm:7.1} suggested by the AI model {\it Gemini}, we did not
pursue this any further.

\section{Proof of \eqref{1.3}}\label{sec:4}

Along with proving the estimate \eqref{1.3}, we are going to show that the
constant $C$ is given by
\begin{equation}\label{4.1}
C=1+\sqrt{2}\sum_{k=1}^{\infty}\frac{(-1)^{k+1}}{(2k-1)8^k}\binom{2k}{k}
\zeta(k-\tfrac{1}{2}),
\end{equation}
where $\zeta(s)$ is the Riemann zeta function.
We begin by rewriting
\begin{align}
\sum_{j=1}^n(-1)^{j+1}\sqrt{jn}
&=\sqrt{n}-\sqrt{n}\sum_{m=1}^{\frac{n-1}{2}}
\left(\sqrt{2m}-\sqrt{2m+1}\right) \label{4.2}\\
&=\sqrt{n}-\sqrt{n}\sum_{m=1}^{\frac{n-1}{2}}
\sqrt{2m}\left(1-\sqrt{1+(2m)^{-1}}\right).\nonumber
\end{align}
We now use the binomial theorem, which in the special case of power 1/2 can be
written as
\begin{equation}\label{4.3}
\sqrt{1-x} = \sum_{k=0}^{\infty}\binom{2k}{k}\frac{x^k}{(1-2k)4^k},\qquad
|x|<1;
\end{equation}
see, e.g., \cite[(1.104)]{Go}. With this identity we get
\begin{equation}\label{4.4}
1-\sqrt{1+(2m)^{-1}} 
=\sum_{k=1}^{\infty}\frac{1}{2k-1}\binom{2k}{k}\left(\frac{-1}{8m}\right)^k,
\end{equation}
and after changing the order of summation we have
\begin{equation}\label{4.5}
\sum_{m=1}^{\frac{n-1}{2}}\sqrt{2m}\left(1-\sqrt{1+(2m)^{-1}}\right)
=\sqrt{2}\sum_{k=1}^{\infty}\frac{(-1)^k}{(2k-1)8^k}\binom{2k}{k}
\sum_{m=1}^{\frac{n-1}{2}}\frac{1}{m^{k-\frac{1}{2}}}.
\end{equation}
We now use the following well-known estimate for the partial sums of $\zeta(s)$:
\begin{equation}\label{4.6}
\sum_{m=1}^N\frac{1}{m^s}=\zeta(s)+\frac{N^{1-s}}{1-s}+\frac{1}{2N^s}
+O\left(\frac{1}{N^{s+1}}\right);
\end{equation}
see, e.g., \cite[pp.~114f]{Ed}. We use \eqref{4.6} with $N=\frac{n-1}{2}$ and
$s=k-\frac{1}{2}$. Then with \eqref{4.1} and \eqref{4.5} the first term on the 
right of \eqref{4.6} immediately gives $1-C$. 

To deal with the second term on the right of \eqref{4.6}, we separate the 
summands for $k=1$ and $k=2$ in \eqref{4.5}, obtaining
\begin{align}
\sqrt{2}\sum_{k=1}^{\infty}&\frac{(-1)^k}{(2k-1)8^k}\binom{2k}{k}
\cdot\frac{\left(\tfrac{n-1}{2}\right)^{\frac{3}{2}-k}}{\tfrac{3}{2}-k}
\label{4.7}\\
&= \sqrt{2}\cdot\frac{-2}{8}\cdot\frac{\sqrt{\tfrac{n-1}{2}}}{1/2}
+\sqrt{2}\cdot\frac{6}{3\cdot 8^2}\cdot\frac{1}{-\tfrac{1}{2}\sqrt{\tfrac{n-1}{2}}}
+O\left(n^{-3/2}\right)\nonumber \\
&= -\frac{1}{2}\sqrt{n-1} - \frac{1}{8}\cdot\frac{1}{\sqrt{n-1}}
+O\left(n^{-3/2}\right)\nonumber \\
&= -\frac{1}{2}\sqrt{n} + \frac{1}{4\sqrt{n}} - \frac{1}{8\sqrt{n}}
+O\left(n^{-3/2}\right),\nonumber
\end{align}
where in the last line we have used the facts that 
\[
\sqrt{n-1}=\sqrt{n}\left(1-\tfrac{1}{2n}\right)+O\left(n^{-3/2}\right),\quad
\frac{1}{\sqrt{n-1}}=\frac{1}{\sqrt{n}}+O\left(n^{-3/2}\right).
\]
Similarly, the third term on the right of \eqref{4.6}, together with 
\eqref{4.5}, gives
\begin{align}
\sqrt{2}\sum_{k=1}^{\infty}&\frac{(-1)^k}{(2k-1)8^k}\binom{2k}{k}
\cdot\frac{1}{2}\cdot\left(\frac{1}{2n}\right)^{\frac{1}{2}-k} \label{4.8}\\
&= \sqrt{2}\cdot\frac{-2}{1\cdot 8}\cdot
\frac{\sqrt{2}}{2\sqrt{n-1}} +O\left(n^{-3/2}\right)
= -\frac{1}{4\sqrt{n}} +O\left(n^{-3/2}\right).\nonumber
\end{align}
Finally, the $O$-term in \eqref{4.6} is easily seen to also be 
$O(n^{-3/2})$. Altogether, \eqref{4.5} with \eqref{4.7} and \eqref{4.8} then
give
\[
\sum_{m=1}^{\frac{n-1}{2}}\sqrt{2m}\left(1-\sqrt{1+(2m)^{-1}}\right)
= 1-C-\frac{\sqrt{n}}{2}-\frac{1}{8\sqrt{n}}+O\left(n^{-3/2}\right).
\]
Substituting this identity into \eqref{4.2}, we obtain the desired estimate
\eqref{1.3}.

\section{Further Remarks}\label{sec:5}

\medskip
{\bf 1.} P\'olya and Szeg{\H o} attribute the identity \eqref{1.1} to 
Bouniakowski \cite{Bou},
whose name is better known in its English transliteration Bunyakovsky.
In \cite{PS2} the identity \eqref{1.1} is the last of several exercises that
use the technique of counting lattice points (Part~8, Problem~20). A more
recent alternative proof was given by Shirali \cite{Sh}.

While the identity \eqref{1.1} does not hold for primes
$p\equiv 3\pmod{4}$, there is in fact an evaluation in this case which involves
the class number of the imaginary quadratic field ${\mathbb Q}(\sqrt{-p})$.
This has been proved in the forthcoming paper \cite{CD}, along with several
other related results, all extending \eqref{1.1}.

\medskip
{\bf 2.} Since the summands in the identity \eqref{7.1} have little or no
arithmetic structure, we cannot expect any classical inversion formula, such as
M\"obius inversion, to apply here. However, the curious identity
\begin{equation}\label{7.21}
\left\lfloor\sqrt{k}\right\rfloor 
= \sum_{d=1}^k\lambda(d)\left\lfloor\frac{k}{d}\right\rfloor,
\end{equation}
which appears in \cite[p.~50]{DD}, could be considered an inversion of the
identity \eqref{7.1}. Here $\lambda(n)$ is the well-known
Liouville function, defined by
\[
\lambda(n) = (-1)^{\Omega(n)},
\]
where $\Omega(n)$ is the number of all prime divisors of $n$, counting
multiplicities. We have not been able to use the identity \eqref{7.21} in the
study of sums such as \eqref{7.1}.

\section{Acknowledgments}

We wish to thank Luc Chamberland for his hearty encouragement to use Google 
Gemini in our explorations. 

\end{document}